\documentclass[12pt, 5p, reqno]{amsart}
\usepackage[all]{xy}
\newcommand{\norm}[1]{{\left\|{#1}\right\|}}    
   \newcommand{\scal}[1]{{\left\langle{#1}\right\rangle}}
    \newcommand{\bz}{\overline{z}}

\usepackage[english]{babel} \usepackage{latexsym, amsfonts, amsmath, amssymb, amsthm, mathrsfs}
\usepackage{color,lmodern} \usepackage{mathpazo} \usepackage[pdftex]{hyperref}
\usepackage[top=2.6cm, bottom=2.6cm, left=2.6cm, right=2.6cm]{geometry}
\numberwithin{equation}{section}  \makeatletter\@addtoreset{equation}{section}
\newtheorem {theorem}{Theorem}[section]
      \newtheorem {definition}[theorem]{Definition}
\newtheorem {corollary}[theorem]{Corollary}     \newtheorem {remark}[theorem]{Remark}   
       \newtheorem {proposition}[theorem]{Proposition}
\newcommand{\C}{\mathbb C} \newcommand{\R}{\mathbb R} 
 	 
\newcommand{\Hq}{\mathbb H}  
\newcommand{\BR}{\mathbb{B}}

\newcommand{\SHyperBTransR}{\mathcal{A}^\alpha_{slice}}

\newcommand{\BergSliceR}{A^{2,\alpha}_{slice}(\BR )}

\begin{document}
\title[]{Bargmann's versus of the quaternionic fractional Hankel transform}

\author{Abdelatif Elkachkouri} 
\author{Allal Ghanmi}
\address{Analysis, P.D.E. $\&$ Spectral Geometry, Lab M.I.A.-S.I., CeReMAR,	Department of Mathematics, P.O. Box 1014,  Faculty of Sciences,	Mohammed V University in Rabat, Morocco}
\email{elkachkouri.abdelatif@gmail.com} 
\email{allal.ghanmi@um5.ac.ma}

\author{Ali Hafoud}
\address{(A.H.) Centre Régional des Métiers de l'Education et de la Formation\newline de kenitra, Morocco} 
\email{hafoudaliali@gmail.com}


\begin{abstract}
We investigate the quaternionic extension of the fractional Fourier transform on the real half-line leading to fractional Hankel transform. This will be handled \`a la Bargmann by means of hyperholomorphic second Bargmann transform for the slice Bergman space of second kind. Basic properties are derived including inversion formula and Plancherel identity. 
\end{abstract}
\maketitle
%
%
%
%
%
%

\large 

\section{Introduction}
 The fractional Fourier transform (FrFT), which is special generalization of the Fourier integral transform, is a powerful tool in many fields of research including mathematics, physics and engineering sciences  \cite{Almeida1994,OzaktasZalevskyKutay2001, KilbasSrivastavaTrujilo2006}.
Its introduction goes back to 1929 when was considered implicitly in Wiener's work \cite{Wiener1929},  discussing the extension of certain results of H. Weyl  and leading later to
Fourier developments of fractional order. Mainly, Wiener sets out to find a one--parameter family of unitary integral operators 
$$ \mathcal{F}_\alpha \varphi (x) := \int_{-\infty}^{+\infty} K_\alpha(x,v) \varphi(v) dv $$
on $L^2(\R)$, for which the $n$--th Hermite function $h_n(x)=H_n(x)e^{-x^2/2}$ is a eigenfunction with $e^{in\alpha}$ as corresponding eigenvalue.
The explicit Wiener formula for the kernel function $K_\alpha$ is a limiting case of Mehler's formula for the Hermite functions. This fact was rediscovered sixty years later in quantum mechanics by Namias \cite{Namias1980}, and showed earlier by H\"ormander \cite{Hormander}.

	Another and elegant way to define the standard FrFT is due to Bargmann 
	\cite{Bargmann1961} (see also \cite{BenahmadiG2019a,GhZine2019,Toft2017}) 
		 and lies on the classical Segal--Bargmann transform $\mathcal{B}$. Indeed, the transform
$ \mathcal{R}_\theta  := \mathcal{B}^{-1} \circ T_\theta \circ \mathcal{B}$, with $\Gamma_\theta f(z) := f(\theta z)$, defines a unitary homeomorphism transform on $L^2(\mathbb{R}^d)$ and satisfies $$\mathcal{R}_\theta h_n(x) = e^{in\alpha} h_n(x).$$
Apparently, except Bargmann's work \cite{Bargmann1961} for FrFT on $L^2(\R^d)$ in connection with the Fock--Bargmann space, there is no general abstract approach in the literature for constructing fractional  transform associated to given invertible integral transform $ \mathcal{S}_{X,Y}:\mathcal{H}_X 
\longrightarrow \mathcal{H}_Y$
 on an arbitrary infinite separable functional Hilbert space $\mathcal{H}_X=L^2(X;\omega_Xd\lambda)$. 
In absence of such general formalism, the authors in \cite{CeleghiniGadelladelOlmo2018} have recently tried to extend, in a partial way, the Fourier fractional formalism to the generalized Laguerre functions by means of their relation to Hermite polynomials.

In the present paper, we provide such abstract formalism \`a la Bargmann. Namely, we deal with integral transforms of the form
$\mathcal{S}_{X,Y}^{-1} \circ T_\theta \circ \mathcal{S}_{X,Y},$
 associated to given special invertible integral transform 
$$ \mathcal{S}_{X,Y} \varphi(y) = \int_X R(x,y) \varphi(x) \omega_X(x) d\lambda(x)$$ 
and given appropriate action of a group $G$.
We show that the performed fractional integral transform inherits numerous properties from the ones of $\mathcal{S}_{X,Y}$.  
The explicit computation shows that the kernel function of this integral transform can be expressed explicitly in terms of the kernel function $R(x,y)$. 
As concrete application, we deal with a special quaternionic fractional Fourier transform (QFrFT) acting on the right quaternionic Hilbert space  $$L^{2,\alpha}_{\Hq}(\R^+)=L^{2}_{\Hq}\left( \mathbb{R}^{+}, x^{\alpha}e^{-x} dx\right)  ,\quad \alpha>0,$$
and associated to the second Bargmann transform for hyperholomorphic Bergman space of second kind \cite{1WAG}.
More precisely, we consider the family of (left) integral transforms 
\begin{equation}\label{IntTrans} 
\mathcal{L}_\theta^\alpha \varphi(y) :=   \int_{0}^\infty  K_\theta^\alpha(x,y) \varphi(x)  dx ,
\end{equation}
whose kernel function can be shown to be given in terms of
the modified Bessel function $I_\alpha$, verifying $ \mathcal{L}_\theta^\alpha(\varphi^{\alpha}_{n}(x))=\theta^{n} \varphi^{\alpha}_{n}(x)$ and leading in particular to a variational definition of  the well-known fractional Hankel transform \cite{Namias1980b,Kerr1991}.  
We also prove that $\mathcal{L}_\theta^\alpha$ is continuous, interpolates continuously the identity operator to the Fourier-Bessel (Hankel) transform and satisfies the index law (semi-group property)
$\mathcal{L}_\theta^\alpha\circ\mathcal{L}^\theta_\eta
=\mathcal{L}^\alpha_{\theta\psi} $, so that the inversion formula reads simply $\left( \mathcal{L}_\theta^\alpha\right) ^{-1}=\mathcal{L}_{\theta^{-1}}^\alpha$. 
The considered family of QFrFT for $L^{2,\alpha}_{\Hq}(\R^+)$ 
appears embedded in a strongly continuous one-parameter
group of unitary operators when $|\theta|=1$.  
The exposition of these ideas in the quaternionic setting add some technical difficulties which we overcome using tools from the theory of slice regular functions.

The paper is organized as follows. In Section 2, we present an abstract formalism for constructing quaternionic fractional integral transforms by means of eigenvalue equation involving orthogonal basis of certain quaternionic Hilbert space.
Section 3 is devoted to the reconstruction of quaternionic fractional Hankel transform  for $L^{2,\alpha}_{\Hq}(\R^+)$ by Bargmann versus, and show how to derive in a simple way their basic properties such as the Plancherel and inversion formulas. 

\section{Preliminaries}
In this section, we review the two classical ways of constructing FrFT
that we extend in a natural way to any arbitrary infinite functional left quaternionic separable Hilbert space $\mathcal{H}_{X}$ on given set $X$. We provide then \`a la Namias a concrete example giving rise to the fractional Hankel transform for the quaternionic left Hilbert space $L^{2,\alpha}_{\Hq}(\R^+)$. 
 
\subsection{Abstract formalism for fractional integral transform}
Let $\mathcal{H}_{X}$ and $\mathcal{H}_{Y}$ be two arbitrary infinite functional left quaternionic separable Hilbert spaces with
orthonoromal bases $\varphi_{n}$ and $\psi_{n}$ on $X$ and $Y$, respectively. The corresponding inner scalar products are given by
$$ \scal{\varphi,\phi}_{\mathcal{H}_{X}} = \int_X \overline{\varphi}(x) \phi(x) \omega_{X}(x) dx$$ 
and 
$$ \scal{\Psi,\Phi}_{\mathcal{H}_{Y}} = \int_Y \overline{\Psi}(y) \Phi(y) \omega_{Y}(y) dy,$$
 respectively, for some weight functions $\omega_{X}$ and $\omega_{Y}$. 
Associated to the data $(X,\mathcal{H}_{X},\varphi_n)$ and $(Y,\mathcal{H}_{Y},\psi_{n})$, we consider  the integral transform
$T_{XY}:\mathcal{H}_{X}\longrightarrow\mathcal{H}_{Y}$ of the form $$T_{XY}(\varphi)(y)=\int_{X} \overline{R(x,y)} \varphi(x) \omega_{X}(x)dx.$$
 We assume that $T_{XY}$ is well defined on $\mathcal{H}_{X}$ such that $T_{XY}(\varphi_{n})=\psi_{n}$. 
 This is equivalent to say that kernel function $R(x,y)$ on  $X\times Y$ can be expanded as  
$$R(x,y)=\sum^{\infty}_{n=0}\varphi_{n}(x)\overline{\psi_{n}(y)}$$
whenever the series in the right-hand side is uniformly and absolutely convergent.
Subsequently, $T_{XY}$ is an invertible integral kernel transform, whose inverse is given by
$$T_{XY}^{-1}\psi(x)=\int_{Y} R(x,y) \psi(y)\omega_{Y}(y)dy$$
for $x \in X$ and $\psi \in\mathcal{H}_{Y}$.
Now, we need to special action of some group $G$ on $Y$ that we will extend to $\mathcal{H}_{Y}$ by considering  $\Gamma:G\times\mathcal{H}_{Y}\longrightarrow\mathcal{H}_{Y}$; $ \Gamma(g,\psi)=\Gamma_{g} (\psi)$, so that the corresponding diagrams 
$$
\def\commutatif{\ar@{}[rd]|{\circlearrowleft}}
\xymatrix{ 
	\mathcal{H}_{X} \ar[r]^{T_{XY}} \ar[d]_{\mathcal{F}_g} \commutatif & \mathcal{H}_{Y} \ar[d]^{\Gamma_g}  \commutatif &\ar[l]_{\Gamma} G \times \mathcal{H}_{Y} \ar[d]^{\Gamma_g}\\
	\mathcal{H}_{X}   & \mathcal{H}_{Y} \ar[l]^{T_{XY}^{-1}}   & G \times \mathcal{H}_{Y}  \ar[l]^{\Gamma} }
$$
becomes commutative, which means that $\Gamma$ satisfies  
$$\Gamma (g,\psi)(y)=\Gamma_{g} (\psi)(y)= \psi(\Gamma_{g}.y); \quad y\in Y, \psi\in \mathcal{H}_{Y}.$$ 
We then perform the fractional transform
$$\mathcal{F}_{rg}=T_{XY}^{-1}\circ\Gamma_g \circ T_{XY}; \quad g\in G.$$
For every $\psi \in\mathcal{H}_{Y}$, we have
\begin{align*}
\mathcal{F}_{rg}(\varphi)(x)
&=\int_{y\in Y} R(x,y) \Gamma_{g} (T_{XY}(\varphi))(y)\omega_{Y}(y)dy
\\
&=\int_{y\in Y} R(x,y) \left(  \int_{x'\in X} \overline{R(x',\Gamma_{g}(y) )}  \varphi(x')\omega_{X}(x')dx'\right) \omega_{Y}(y)dy
\\&\stackrel{Fubini}{=} \int_{x'\in X}\widetilde{R_{g}}(x',x)\varphi(x')\omega_{X}(x')dx',
\end{align*}
where $R_{g}(x',x)$ stands for
$$\widetilde{R_{g}}(x',x)= \scal{ R(x',\Gamma_{g}(y) ),R(x,y) }_{\mathcal{H}_{Y}}.$$
An expansion  of $\widetilde{R_{g}}(z,x)$, at least formally, is the following 
\begin{align}
\widetilde{R_{g}}(x',x) 
&=\sum^{\infty}_{m=0}\sum^{\infty}_{n=0} \varphi_{n}(x)\langle \psi_{n} ,\Gamma_{g}\psi_{m} \rangle_{\mathcal{H}_{Y}} \overline{\varphi_{m}(x')}\nonumber
\\& =\sum^{\infty}_{n=0}\varphi_{n}(x)\chi_{n}(g) \, \overline{\varphi_{n}(x')} =: R_{g}(x',x). \label{series}
\end{align}
The last equality follows under the additional assumption  that 
$$\Gamma_{g}\psi_{m}(y) = \psi_{m}(\Gamma_{g}(y)) =\chi_{m}(g)\psi_{m}(y).$$ 
According to the above discussion, we reformulate the following definitions.

\begin{definition}\label{Def1FrFT}
	If the series in the right-hand side of \eqref{series} 
	converges absolutely and uniformly to $R_{g}(x',x)$, then
	$$ \mathcal{F}_{rg} (\varphi )(x) := 
	\int_{x'\in X} R_{g}(x',x)\varphi(x')\omega_{X}(x')dx'    $$
		defines a like-fractional Fourier transform for the data $(\mathcal{H}_{X},\varphi_n,\chi_n)$. 
\end{definition}

\begin{definition}\label{Def2FrFT} 
	We call fractional Fourier transform associated to $T_{XY}$ and $\Gamma$ the integral transform
	$$
	 \widetilde{\mathcal{F}_{rg}} (\varphi )(x)  = T_{XY}^{-1}\Gamma_{g}  T_{XY}  (\varphi )(x)  =  \int_{x'\in X} \widetilde{R_{g}}(x',x)\varphi(x')\omega_{X}(x')dx'  $$
	with  
	\begin{align}\label{KernelInt}
	\widetilde{R_{g}}(x',x) = \scal{R(x',\Gamma_{g}(y) ),R(x,y)}_{\mathcal{H}_{Y}}.
	\end{align}
\end{definition}

\begin{remark}
	We have $\mathcal{F}_{rg}(\varphi_{n})=\varphi_{n}\chi_{n}(g)$. This gives an integral representation for $\varphi_{n}$.
\end{remark}

\subsection{\textbf{Quaternionic fractional Hankel transform (\`a la Namias)}}

We consider  the right quaternionic Hilbert space  $L^{2,\alpha}_{\Hq}(\R^+)$; $\alpha >0$,
of all  quaternionic-valued functions on the half real line $\R^+$ that are square integrable with respect to the inner product 
$$\scal{\varphi,\psi}_{\alpha}=\int_{\mathbb{R}^{+}}\overline{\varphi(x)}\psi(x)
 x^{\alpha}e^{-x}dx 
. $$
We denote by $\norm{\cdot}_{\alpha}$ the associated norm.
A complete orthonormal system in $L^{2,\alpha}_{\Hq}(\R^+)$ is given by the functions 
\begin{equation} \label{basisLaguerre}
\varphi_{n}(x) =\left( \frac{n!}{\Gamma(\alpha +n+1)}\right)^{1/2} L ^{(\alpha)}_{n}(x),
\end{equation}
where $L^{(\alpha)}_{n}(x)$ denotes the generalized Laguerre polynomials 
\begin{align}\label{LaguerrePoly}
L^{(\alpha)}_{n}(x) 
= \frac{x^{-\alpha }e^{x}}{n!}\frac{d^{n}}{dx^{n}}\left(x^{n+\alpha } e^{-x}\right).
\end{align}

Accordingly,  the series function in Definition \ref{Def1FrFT} reduces further to  Hille--H+ardy identity \cite[(6.2.25) p. 288]{AndrewsAskeyRoy1999}
\begin{align}
R_\theta^\alpha(x,y) &= \sum_{n=0}^{+\infty} \frac{n! }{\Gamma(\alpha +n+1)}  \theta^n   L^{(\alpha)}_{n}(x) L^{(\alpha)}_{n}(y) \nonumber
\\&  =
\frac{1}{1-\theta}
\left(\frac{1}{\theta x y} \right)^{\alpha/2}
\exp\left(-\frac{\theta (x+y)}{1-\theta}\right) I_\alpha\left( \frac{2\sqrt{\theta}}{1-\theta} \sqrt{xy}\right) \label{Kernel-Hill} 
\end{align}
valid for $|\theta[<1$ and nonnegative integer $\alpha$,  where $I_\alpha (\xi)$ denotes the modified Bessel function \cite[p.199]{AndrewsAskeyRoy1999}
$$ I_\alpha (\xi) =  \left(  \frac{\xi}{2}\right)^{2n+\alpha} \sum_{n=0}^{\infty} \frac{1}{n! \Gamma(\alpha+n+1)} \left(  \frac{\xi}{2}\right)^{2n}  .$$
Thus, we can rewrite the kernel function 
$K_\theta^\alpha(x,y) = x^\alpha e^{-x}  R_\theta^\alpha(x,y)$  as
\begin{align}\label{KernelBessel} 
K_\theta^\alpha(x,y) &=
\frac{1}{1-\theta} \left(\frac{x}{\theta y} \right)^{\alpha/2}
\exp\left(-\frac{x+\theta y}{1-\theta}\right) 
I_\alpha\left( \frac{2\sqrt{\theta}}{1-\theta} \sqrt{xy}\right) ,
\end{align}
so that the corresponding integral operator is well--defined on $L^{2,\alpha}_{\Hq}(\R^+)$  by 
\begin{align} \label{FrHT}
\mathcal{L}_\theta^\alpha(\varphi)(y)=\int_{0}^{+\infty}K^\alpha_{\theta}(x,y) \varphi(x)dx.
\end{align}
Such transform is closely connected to the fractional Hankel transform \cite{Namias1980b,Kerr1991}.
The Laguerre polynomial $\varphi^{\alpha}_{n}(x)$ in \eqref{basisLaguerre} is (left) eigenfunction of $\mathcal{L}_\theta^\alpha$ with $\theta^{n}$ as corresponding (right) eigenvalue, 
$$ \mathcal{L}_\theta^\alpha(\varphi^{\alpha}_{n}(x))= \varphi^{\alpha}_{n}(x)\theta^{n} .$$ This readily follows from the definition of $\mathcal{L}_\theta^\alpha$.
Moreover, we assert

\begin{proposition}
	For $|\theta|<1$, the integral transform $\mathcal{L}_\theta^\alpha$ defines a continuous $k$--contraction 
	from $L^{2,\alpha}_{\Hq}(\R^+) $ into itself with $k\leq (1-|\theta|^2)^{-1/2}$.
\end{proposition}

\begin{proof} 
	Since $(\varphi^{\alpha}_{n})_n$ in \eqref{basisLaguerre} is a complete orthonormal system in $L^{2,\alpha}_{\Hq}(\R^+)$, we can expand any $f\in L^{2,\alpha}_{\Hq}(\R^+) $ as $\displaystyle f(x)= \sum_{n=0}^{+\infty} \varphi^{\alpha}_{n} c_n$ for some $c_n\in\Hq$.
	Hence, using the fact that
	$\mathcal{L}_\theta^\alpha(\varphi^{\alpha}_{n})(x)
	= \varphi^{\alpha}_{n}(x) \theta^n ,     $
	we get 
	\begin{align}
	\mathcal{L}_\theta^\alpha (f)(x)= \sum_{n=0}^{+\infty}  \varphi^{\alpha}_{n}(x) \theta^n c_n.
	\end{align}
	Using the orthogonality of $\varphi^{\alpha}_{n}$, we obtain 		
	\begin{align*}
	\norm{\mathcal{L}_\theta^\alpha (\varphi)}^2 
	&= \sum_{n=0}^{+\infty}  |\theta|^{2n} |c_n|^{2}
\\	&\leq \left( \sum_{n=0}^{+\infty}  |\theta|^{2n} \right) \left( \sum_{n=0}^{+\infty} |c_n|^{2n}\right) 
\\	&\leq \left(\frac{1}{1-|\theta|^2} \right) \norm{\varphi}^2
	\end{align*}
	which requires $|\theta|<1$. 
\end{proof}
 
\section{\textbf{The QFrFT for $L^{2,\alpha}_{\Hq}(\R^+)$}}

In view of the explicit expression of the kernel function in \eqref{KernelBessel}, we see that we can consider the limit case of the Hille--Hardy formula which corresponds to $|\theta|=1$ with $\theta \ne 1$. 
We show below that this can recovered by the formalism presented in Definition \ref{Def2FrFT} and specified for $L^{2,\alpha}_{\Hq}(\R^+)$, so that for $\theta= 1$ the considered transform reduces further to the identity operator 
of $L^{2,\alpha}_{\Hq}(\R^+)$. To this end, we begin by recalling that the hyperholomorphic second Bargmann integral transform \cite{1WAG}, defined by
\begin{equation}\label{IntTransf}
[\SHyperBTransR \varphi](q)=  \frac{1}{\sqrt{\pi \Gamma(\alpha)}\left( 1-q\right)^{\alpha +1}}  \int_{0}^{+\infty}   \exp\left(\frac{tq}{q-1}\right) \varphi(t) t^\alpha  e^{-t} dt,
\end{equation}
is the quaternionic analogue of the complex second Bargmann transform introduced by Bargmann himself in \cite[p.203]{Bargmann1961}.
It establishes a unitary isometric from $L^{2,\alpha}_{\Hq}(\R^+)$ onto the slice hyperholomorphic Bergman space (of second kind)  
on the unit ball $\mathbb{B}$ in $\mathbb{R}^{4}$,  
\begin{align}\label{HypBspace}
A^{2,\alpha}_{slice}(\mathbb{B}) :=   \mathcal{SR}(\mathbb{B}) \cap L^{2,\alpha}(\mathbb{B}_I),
\end{align} 
where $I\in \Hq$ with $I^2=-1$; $\mathbb{B}_I = \mathbb{B}\cap \C_I$  and
$$ L^{2,\alpha}(\mathbb{B}_I):= 
\left\{ f: \mathbb{B} \longrightarrow \Hq;\ \int_{\mathbb{B}_I} |f(z)|^{2}d\lambda^{\alpha}_I(z) <+ \infty\right\}.$$
Here $d\lambda^{\alpha}_I$ denotes the Bergman measure on the unit disc $\mathbb{B}_I$ in $\mathbb{R}^{2}$ given by
$$
d\lambda^{\alpha}_I(z=x+Iy) = \left(1-x^2-y^2 \right)^{\alpha-1} dxdy .$$ 
Sequentially, we have 
$$ A^{2,\alpha}_{slice}(\mathbb{B}) = \left\{ f(q) = \sum_{n=0}^\infty q^n c_n; \, c_n\in \Hq, \, \sum_{n=0}^\infty \frac{n!}{\Gamma(\alpha+n+1)} |c_n|^2 <+\infty\right \},$$
so that the restriction to $\mathbb{B}_I$ is the classical Bergman space on the unit disc of $\C_I$. It should be mentioned here that the scaler product defining $L^{2,\alpha}(\mathbb{B}_I)$, 
$$ \scal{f,g}_I := \int_{\mathbb{B}_I} \overline{f(z)} g(z)  d\lambda^{\alpha}_I(z),$$ 
is independent of $I$ when acting on $A^{2,\alpha}_{slice}(\mathbb{B}) \times A^{2,\alpha}_{slice}(\mathbb{B})$, i.e., 
$ \scal{f,g}_I = \scal{f,g}_J$
for any $f,g\in A^{2,\alpha}_{slice}(\mathbb{B})$ and any $I,J$ such that $I^2=J^2=-1$. 

The inverse of the second Bargmann transform  $\SHyperBTransR$ is well--defined from $\BergSliceR$ onto $ L^{2,\alpha}_{\Hq}(\R^+)$, and is given by \cite{1WAG}
\begin{align}\label{Inverse2BTq}
[\SHyperBTransR]^{-1} f(t) =  \frac{1}{\sqrt{\pi \Gamma(\alpha)}}  \int_{B_I}
\exp \left( \frac{t\overline{z}}{\overline{z}-1}\right) \frac{(1-|z|^2)^{\alpha-1}}{\left(1-z\right)^{\alpha+1}}   f|_{B_I}(z)   dxdy.
\end{align}

Notice for instance that the definition of $\BergSliceR$ is based on the classical one on a given disc $\mathbb{B}_I$. This was possible by extending the complex holomorphic functions to the whole $\mathbb{B}$ by the representation formula (see for example \cite{ColomboSabadiniStruppa2016}). While the transform $\SHyperBTransR$ in \eqref{IntTransf} is associated to the kernel function 
\begin{equation}\label{expKernel}
A^\alpha_{slice} (x;q) :=   \frac{1}{\sqrt{\pi \Gamma(\alpha)}\left( 1-q\right)^{\alpha +1}}  \exp\left( \frac{xq}{q - 1}\right)
\end{equation}
on $\R^{+} \times \mathbb{B} $, and obtained as bilinear generating function involving the functions  $(\varphi_{n}^\alpha)_n$ in \eqref{basisLaguerre} and the orthonormal basis of $\BergSliceR$ given by the functions
\begin{equation} \label{onbA}
f_n(q) =  \left(\frac{\Gamma(n+\alpha+1)}{\pi \Gamma(\alpha) n!}\right)^{1/2} q^{n}.
\end{equation} 

Now, by means of $\SHyperBTransR$, its inverse $[\SHyperBTransR]^{-1}$ and the angular unitary operator 
$\Gamma_{\theta}(f)(q)=f(q \theta )$,
we perform the transform 
\begin{align}\label{limitoperator}
\widetilde{\mathcal{L}_\theta^\alpha}:=[\SHyperBTransR]^{-1} \Gamma_{\theta} \SHyperBTransR
\end{align}
on $L^{2,\alpha}_{\Hq}(\R^+)$. Here we consider the
$\Gamma_\theta$--action  $\Gamma_{\theta}q=q \theta$ of $G=U_\Hq(1)$ on $\mathbb{B}$, 
that we extend to the hyperholomorphic Bergman space $\BergSliceR$ by considering 
\begin{align}\label{star}
\Gamma_{\theta}(f)(q)=f_{\star}(q \theta ) := \sum_{n=0}^\infty q^n \theta^n c_n
\end{align}
for given $f(q)= \sum_{n=0}^\infty q^n c_n\in \BergSliceR$. 
The function $q\longmapsto f_{\star}(q\theta )$ is in fact  the slice regularization of $q\longmapsto f(q \theta )$ obtained by making use of the left $\star^L_s$-product for left slice regular functions  $\displaystyle  f (q) = \sum_{n=0}^\infty q^n a_n$ and $\displaystyle  g (q) = \sum_{n=0}^\infty q^n b_n $ on $\Hq$ defined by \cite{GentiliStoppato08}
\begin{equation}\label{starProduct}
(f\star^L_s g) (q) = \sum_{n=0}^\infty q^n \left(\sum_{k=0}^n a_k b_{n-k}\right).
\end{equation}
 In particular, we have
\begin{equation}\label{fntheta} (f_n)_{\star}(q\theta) := f_n(q)\theta^n,
\end{equation}
and therefore we may prove the following.

\begin{proposition}
	For $\theta\in\Hq$ with $|\theta|\leq 1$, the transform
	$\widetilde{\mathcal{L}_\theta^\alpha}$ in \eqref{limitoperator}
	defines a continuous integral transform from $L^{2,\alpha}_{\Hq}(\R^+)$ onto $L^{2,\alpha}_{\Hq}(\R^+)$ with norm not exceed $1$. For $|\theta|=1$, we have 	$$ \scal{\widetilde{\mathcal{L}_\theta^\alpha} \varphi, \widetilde{\mathcal{L}_\theta^\alpha} \psi}=\scal{ \varphi,\psi} .$$
\end{proposition}

\begin{proof} The operator $\widetilde{\mathcal{L}_\theta^\alpha}$ in \eqref{limitoperator} is well--defined from $L^{2,\alpha}_{\Hq}(\R^+)$
	into itself if and only if the action $\Gamma_{\theta}$ leaves the space $\BergSliceR$ invariant, which is clear from the definition of $\Gamma_{\theta}$ given through \eqref{star}. Moreover, using the fact $\SHyperBTransR\varphi^{\alpha}_{n} = f_n$ as well as \eqref{fntheta}, we get 
	\begin{align*}
	\widetilde{\mathcal{L}_\theta^\alpha} (\varphi^{\alpha}_{n}(y))
	&=[\SHyperBTransR]^{-1}\left( f_n (\cdot) \theta^n \right) (y)
	=\varphi^{\alpha}_{n}(y) \theta ^{n}.
	\end{align*}
	In addition, under the  condition that $|\theta|=1$, it is clear that $\Gamma_\theta$  preserves the scalar product in $\BergSliceR$.
	Indeed, for every 
	$f=\sum_{n=0}^\infty f_n   c_n \quad \mbox{and} \quad  g=\sum_{n=0}^\infty f_n d_n \in \BergSliceR,$ we have 
	\begin{align*}
	\scal{ \Gamma_{\theta} f,  \Gamma_{\theta}  g}_{\BergSliceR}
	& 	= \sum_{n,m=0}^\infty	\overline{c_n} 
	\overline{\theta^n} \scal{  f_n ,   f_m}_{\BergSliceR}  \theta^m d_m
	\\&	= \sum_{n=0}^\infty  \overline{c_n}|\theta^n|^2 d_n
	\\&	=  \scal{  f, g}_{\BergSliceR}
	.\end{align*}
	Accordingly, the identity $ \scal{\widetilde{\mathcal{L}_\theta^\alpha} \varphi, \widetilde{\mathcal{L}_\theta^\alpha} \psi}=\scal{ \varphi,\psi}$ 
	follows as composition of operators preserving scaler product. 
\end{proof}

\begin{corollary} 
	If $|\theta|=1$, then the QFrFT $\widetilde{\mathcal{L}_\theta^\alpha}$ defines a unitary transform from $L^{2,\alpha}_{\Hq}(\R^+)$ into $L^{2,\alpha}_{\Hq}(\R^+)$.
\end{corollary}

\begin{remark}
	The family of one--parameter transforms $\mathcal{L}_\theta^\alpha$ verifies the semi-group property $\mathcal{L}_\theta^\alpha\circ\mathcal{L}^\alpha_\eta
	=\mathcal{L}^\alpha_{\theta\eta} $, so that its inverse  is $\mathcal{L}_{1/\theta}$ when $\theta\ne 0$. But we do not have  $\mathcal{L}_\theta^\alpha\circ\mathcal{L}^\alpha_\eta
	= \mathcal{L}^\alpha_\eta \circ \mathcal{L}_\theta^\alpha$ in general,
	for lack of commutativity in $\Hq$.   However, 
	$\mathcal{L}_\theta^\alpha\circ\mathcal{L}^\alpha_\eta
	=\mathcal{L}^\alpha_{\theta\eta} = \mathcal{L}^\alpha_\eta \circ \mathcal{L}_\theta^\alpha$
	holds only when $\theta$ and $\psi$ belongs to the same slice $\C_I:=\R+I\R\subset \Hq$; $I^2=-1$. 
\end{remark}

The next result gives the explicit expression of the inverse of $\widetilde{\mathcal{L}_\theta^\alpha}$.  

\begin{proposition}
	For any quaternionic $\theta \ne 0$, the inverse of $\widetilde{\mathcal{L}_\theta^\alpha}$ is given by 
	$$ (\widetilde{\mathcal{L}_\theta^\alpha})^{-1}=\mathcal{A}^{-1}\Gamma_{\theta}^{-1} \mathcal{A}=\widetilde{\mathcal{L}_{\frac{1}{\theta}}}. $$
\end{proposition}

\begin{proof} 
It is immediate form the definition of $\widetilde{\mathcal{L}_\theta^\alpha}$ and the fact that $\Gamma_{\theta}\circ\Gamma_{\eta}=\Gamma_{\theta\eta}$. 
\end{proof}

The following result identifies the kernel function given by \eqref{KernelInt},
\begin{align}\label{expKerneltheta1}
\widetilde{R_{\theta}^\alpha}(x,y) :=  \scal{ A^\alpha_{slice} (x;  \Gamma_\theta\cdot)    , A^\alpha_{slice} (y;\cdot)  }_{L^{2,\alpha}(\mathbb{B}_I)}
\end{align}
of the QFrFT  transform 
\begin{align*}
[\widetilde{\mathcal{L}_\theta^\alpha} (\varphi )](y)&= \scal{  \widetilde{R_{\theta}^\alpha}  , \varphi }_{L^{2,\alpha}_{\Hq}(\R^+)} .
\end{align*}

\begin{theorem}
	The kernel function $\widetilde{R_{\theta}^\alpha}(x,y)$ is a left slice regular and coincides with the kernel function of the fractional Hankel transform on the quaternionic unit ball.
	 Moreover, the explicit expression of $\widetilde{\mathcal{L}_\theta^\alpha}$ is given by 
	\begin{align}\label{limitcase}	
\widetilde{\mathcal{L}_\theta^\alpha}  \varphi(y) = 
\frac{e^{\frac{\theta y}{\theta-1}}} {(1-\theta) (\theta y)^{\alpha/2}} 
\int_{0}^\infty  x ^{\alpha/2} I_\alpha\left( \frac{2\sqrt{\theta}}{(1-\theta)} \sqrt{xy}\right)
 e^{\frac{x}{\theta-1}}
\varphi (x)  dx
\end{align}
for any $\theta\in\Hq$ with $|\theta|\leq 1$ and $\theta\ne 1$.	
\end{theorem}

\begin{proof}
		Notice first that for $\theta=1$ there is nothing to prove since in this case, the operator $\widetilde{\mathcal{L}_\theta^\alpha}$ reduces further to the identity operator  of the Hilbert space $L^{2,\alpha}_{\Hq}(\R^+)$ and 
	the $R_{\theta}^\alpha(x,y)$ in \eqref{expKerneltheta1} is to considered as the Dirac delta function. 
	To identify the closed expression of the kernel  $\widetilde{R_{\theta}^\alpha}(x,y)$, we should notice  that the $\Gamma_\theta$-action reads
	$$ \Gamma_\theta (q\longmapsto A^\alpha_{slice} (x; q) ) 
	=  \left( 1 -  q \theta\right)^{-\alpha-1} \star \exp_{\star}\left( x q \theta  , [q \theta - 1]^{-1} \right) ,$$
	where 
	$$ \exp_{\star}\left( f(q) , g(q)\right)= 
	\sum_{n=0}^\infty \frac{f^{n\star}(q) \star  g^{n\star}(q)}{n!} .$$
	For $\theta$ being a non real quaternionic number, there exists a unique imaginary unit $I_\theta$; $I_\theta^2=-1$, such that $\theta \in \C_{I_\theta}\cap S^3$. 
	By means of \eqref{expKerneltheta1} and the independence of the scaler product $ \scal{f,g}_I$ in $I$ when acting on $A^{2,\alpha}_{slice}(\mathbb{B}) $, we may write  
	\begin{align*}
\widetilde{R_{\theta}^\alpha}(x,y) &:=  \scal{ \Gamma_\theta A^\alpha_{slice} (x; \cdot)   ,  A^\alpha_{slice} (y;\cdot)  }_{L^{2,\alpha}(\mathbb{B}_{I_\theta})} 
\\&= \frac 1{\pi\Gamma(\alpha)} \int_{B_I}  \frac{ \exp \left( \frac{x z \theta}{ z \theta - 1}\right)
		\exp\left( \frac{y \bz}{\bz - 1}\right) }{( 1 -  z \theta)^{\alpha+1} 
		( 1 - \bz )^{\alpha+1}}
	\left(1-\left|z\right|^2\right)^{\alpha-1}d\lambda_I(z)
	\end{align*}
	in view of the explicit expression of the kernel function $A^\alpha_{slice}$ in \eqref{expKernel}. Using the generating function for generalized Laguerre polynomials \cite[p.288]{AndrewsAskeyRoy1999} 
	$$(1-z)^{-\alpha-1}\exp\left(\frac{xz}{z-1}\right)=\sum_{n=0}^{\infty}L^{(\alpha
		)}_{n}\left(x\right)z^{n},$$
	provided $|z|<1$, as well as Fejer's formula  \cite[Theorem 8.22.1, p. 198]{Szego1975},  it is not hard to see that the involved $z$-function series are uniformly convergent on any compact set contained in unit disk. Therefore, direct computation yields
	\begin{align}
	\widetilde{R_{\theta}^\alpha}(x,y) 
		&=\frac 1{\pi\Gamma(\alpha)}\int_{\mathbb{D}}\left( \sum _{n=0}^\infty L^{(\alpha)}_{n}(x) \theta^n z^{n}\right) \left( \sum _{m=0}^\infty L^{(\alpha)}_{m}(y) \bz^{m}\left(1-| z|^{2}\right)  ^{\alpha -1}\right)  d\lambda(z)  \nonumber
	\\	&=\frac 1{\pi\Gamma(\alpha)}\sum_{n=0}^\infty\sum _{m=0}^\infty \theta^n L^{(\alpha)}_{n}(x) L^{(\alpha)}_{m}(y) \int_{\mathbb{D}}   z^{n} \bz^{m}\left(1-| z|^{2}\right)  ^{\alpha -1} d\lambda(z) \nonumber
	\\&=  \sum_{n=0}^\infty  \frac{n!}{\Gamma(n+1+\alpha)}  \theta^n L^{(\alpha)}_{n}(x) L^{(\alpha)}_{n}(y)\label{HH}
	\end{align}
	for $|\theta z| <1$ which holds true when $|\theta| \leq 1$ and $|z|<1$.
	This provides the expansion series of the restriction of $\widetilde{R_{\theta}^\alpha}(x,y)$ to any $\mathbb{B}_I$.
	For $|\theta|<1$, we recognize the Hille--Hardy identity \eqref{Kernel-Hill} for Laguerre polynomials. 
	Thus, we have
	\begin{align} \label{equalitySlice}
		\widetilde{R_{\theta}^\alpha}(x,y)=\frac{1}{(1-\theta) }  \left( \frac{1}{xy\theta}\right)^{\alpha/2}
		I_{\alpha}\left( \frac{2\theta^{1/2}}{1-\theta} \sqrt{xy}\right)\exp \left(- \frac{\theta(x+y)}{1-\theta} \right) 
		\end{align}
		for $|\theta|<1$ and $\theta \notin \R$. 
		 This leads to \eqref{FrHT} by considering the kernel function $
		\widetilde{R_{\theta}^\alpha}(x,y) x^\alpha e^{-x}$.
The right-hand side in \eqref{equalitySlice} is clearly a slice regular 
		function in $\theta\in \mathbb{B}$ for $x,y$ being reals.
	The extension of \eqref{equalitySlice} to the whole unit open ball $\mathbb{B}$	relies on the Identity Principle for left slice regular functions \cite{GentiliStoppato08}, since both 
	sides of \eqref{equalitySlice} are left slice regular and 
	coincide at least on the upper half unit ball.
To conclude, we need only to examine the validity of the closed expression in  the right-hand side of \eqref{equalitySlice} for the expansion of $\widetilde{R_{\theta}^\alpha}(x,y)$ with remains valid when  $|\theta|=1$ with $\theta \ne 1$. 
This can be handled by  fixing $\theta$ such as and let $\varepsilon \in (0,1)$, so that \eqref{equalitySlice} holds true for $|\varepsilon\theta| <1$,
		and next sending $\varepsilon$ to $1^-$, at least formally. This can be rigorously justified making use of test functions  and classical argument from the Schwartz theory of distributions. 
\end{proof}

\begin{remark}
	By taking $\theta=-1$ with $\sqrt{\theta}=i$ in \eqref{limitcase}, we recover the classical Fourier-Bessel transform \cite{Namias1980b,Kerr1991}
	$$ \left( \mathcal{H}_\alpha \psi\right)  (y) :=   
	\int_{0}^\infty  u J_\alpha\left( y u \right)  \psi (u)  du$$
	for $\psi \in L^2(\R^+)$, where $J_\alpha$ is the Bessel function.
	 Indeed, by setting  $\widetilde{\mathcal{L}^\alpha}=\widetilde{\mathcal{L}_{-1}^\alpha}$ and making the change of variable $u^2 =x$ and the function $\psi(u)= x^{\alpha/2} e^{-x/2} \varphi(x)= u^\alpha e^{-u^2/2} \varphi(u^2) $ we get 
		\begin{align*}	
\widetilde{\mathcal{L}^\alpha}  \varphi(y^2) &= 
\frac{e^{y^2/2}} { i^{\alpha} y^{\alpha}} 
\int_{0}^\infty   u ^{\alpha+1} I_\alpha\left( iy u \right) e^{-u^2/2} \varphi (u^2)  du 
= \frac{e^{y^2/2}} { y^{\alpha}} \left( \mathcal{H}_\alpha \psi\right)  (y).
\end{align*} 
The last equality follows since $I_\alpha(x) = i^{-\alpha} J_\alpha(i x)$.
\end{remark}

\begin{remark}
	The considered family of QFrFT on the real half-line appears embedded in a strongly continuous one-parameter
	group of unitary operators the quaternionic context.  
	Moreover, it is continuous and interpolates continuously the identity operator ($\theta=1$) to the Hankel transform \cite[p. 216]{AndrewsAskeyRoy1999} 
	corresponding to $\theta=-1$. 
\end{remark}

\begin{remark}
The considered transform can be used to reintroduce the hyperholomorphic Bergman space $
A^{2,\alpha}_{slice}(\mathbb{B})$ in \eqref{HypBspace} as well as some of their specific generalization in the context of slice regular functions on the unit quaternionic ball by considering the dual transform of $\theta \longmapsto \widetilde{\mathcal{L}_\theta^\alpha}\varphi(y)$, for fixed $y\in (0,+\infty)$. For the limit case of $y=0$, the last transform is nothing than the Bargmann transform in \eqref{IntTransf}. 
\end{remark}

\end{document}